\RequirePackage{fix-cm}
\documentclass[smallextended]{svjour3}       
\smartqed  

\usepackage{amssymb}
\usepackage{ifthen}
\usepackage{tikz}
\usepackage{graphicx}
\usepackage{hyperref}

\makeatletter
\def\@seccntformat#1{\csname the#1\endcsname.\ } 
\makeatother

\begin{document}

\title{Distance-$2$ MDS codes and latin colorings in the Doob graphs%
\thanks{This is the accepted version of the paper in Graphs and Combinatorics, Vol.~34, Iss.~5, pp. 1001--1017, 2008;
doi \href{10.1007/s00373-018-1926-4}{https://doi.org/10.1007/s00373-018-1926-4} \copyright~Springer Japan KK, part of Springer Nature 2018}
\thanks{The work was funded by the Russian Science Foundation (grant No 14-11-00555).}
}

\titlerunning{Distance-$2$ MDS codes and latin colorings in the Doob graphs}        

\author{Denis S. Krotov \and Evgeny A. Bespalov 
}


\institute{D. Krotov \at
              Sobolev Institute of Mathematics, pr. Akademika Koptyuga 4, Novosibirsk 630090, Russia \\
              Tel.: +7-383-329-75-42\\
              \email{krotov@math.nsc.ru}           \\
               orcid 0000-0002-8516-755X
           \and
           E. Bespalov \at
              Sobolev Institute of Mathematics, pr. Akademika Koptyuga 4, Novosibirsk 630090, Russia\\
              \email{bespalovpes@mail.ru}\\
              orcid 0000-0001-9484-0091
}

\date{Received: 2016-11-15 / Accepted: 2018-07-11}

\maketitle

\begin{abstract}
 The maximum independent sets in the Doob graphs $D(m,n)$ are analogs of the distance-$2$ MDS codes in Hamming graphs and of the latin hypercubes. We prove the characterization of these sets stating that every such set is semilinear or reducible. As related objects, we study vertex sets with maximum cut (edge boundary) in $D(m,n)$ and prove some facts on their structure. We show that the considered two classes (the maximum independent sets and the maximum-cut sets) can be defined as classes of completely regular sets with specified $2$-by-$2$ quotient matrices. It is notable that for a set from the considered classes, the eigenvalues of the quotient matrix are the maximum and the minimum eigenvalues of the graph. For  $D(m,0)$, we show the existence of a third, intermediate, class of completely regular sets with the same property.
\end{abstract}

\def\Sh{\mathrm{Sh}}

\def\MMDS{$2\times$MDS }
\def\VV{{\scriptscriptstyle\mathrm{V}}}
\def\EE{{\scriptscriptstyle\mathrm{E}}}

\newcommand{\MMMDS}{$1\frac12\times$MDS }

\def\shpart#1 #2 #3 #4!{node [#1] {} +(1,0) node [#2] {} +(2,0) node [#3] {} +(3,0) node [#4] {}}
\def\sh#1#2#3#4{\begin{tikzpicture}[
scale=0.7,
nz/.style={circle,fill=white,draw=black, 
           inner sep=2.5pt},
xz/.style={circle,fill=black!50!white,draw=black, 
           inner sep=2.5pt},
vz/.style={circle,fill=black!23!white,draw=black, 
           inner sep=2.5pt},
wz/.style={circle,fill=black!66!white,draw=black, 
           inner sep=2.5pt},
zz/.style={circle,fill=black,draw=black, 
           inner sep=2.5pt},
scale=0.7]
\begin{scope}
\clip [xslant=-0.577] (-1.4,-1.20) rectangle (2.4,2.1);
\draw[xslant=0.577,ystep=.866,xstep=1,draw=black] (-4.9,-2.1) grid (5.4,3.9);
\draw[xslant=-0.577,ystep=9.866,xstep=1,draw=black] (-3.4,-2.1) grid (6.4,3.9);
\draw (-120:1) \shpart #4!
++(120:1) \shpart #3!
++(120:1) \shpart #2!
++(120:1) \shpart #1!;
\end{scope}
\end{tikzpicture}}

\section{Introduction}\label{s:intro}

In this paper, we characterize the maximum independent sets in the Doob graphs.
The Doob graph $D(m,n)$, as a distance-regular graph, has the same parameters as the Hamming graph $H(2m+n,4)$.
As we will see, 
the maximum independent sets in the Hamming graphs share many properties with those sets in the Doob graphs.
This conclusion was already suggested
by the existence of the injective map 
from the class of maximum independent sets in $D(m,n)$
to the class of maximum independent sets in $H(2m+n,4)$ \cite{Kro:2015:N-MDS-Doob}. However, there are some difficulties
with the use of that injection for the characterization of the former class,
and the authors of the current paper decided to use another way to prove the characterizing theorem, presented in Section~\ref{s:res}.

The maximum independent sets in the Hamming graphs are studied in different areas of mathematics.
In coding theory, they are the distance-$2$ MDS codes (these codes do not correct any errors, 
but they are used in the construction of codes with larger code distance).
In combinatorics, these sets are known as the latin hypercubes, 
multidimensional generalizations of the latin squares, which correspond to the case $n=3$ 
(one of the coordinate is usually considered as dependent from the others).
In nonassociative algebra, an $n$-ary quasigroup is exactly 
a pair of a set and an $n$-ary operation over this set 
whose value table is a latin hypercube.
Every maximum independent set in a Hamming graph is a completely regular code of radius $1$;
the nontrivial eigenvalue of this code is the minimum eigenvalue of the graph.

For fixed $q$, the class of the distance-$2$ MDS codes in $H(N,q)$ is described for $q\le 4$. 
For $q\le 3$, the description is rather simple as for each $n$ there is only one such set, up to equivalence.
In the case $q=4$, there are $2^{2^{N+o(N)}}$ nonequivalent MDS codes in $H(N,4)$;
however, there is a constructive description of the class of these codes \cite{KroPot:4}.
Notably, the case $q=4$ is a special case for the Hamming graphs $H(N,q)$ 
from the point of view of algebraic combinatorics:
$H(N,4)$, $N\ge2$ are the only Hamming graphs that are not defined as distance-regular graphs with given parameters.
The distance-regular graphs with the same parameters as $H(N,4)$ are the Doob graphs $D(m,n)$, $2m+n=N$.

The main goal of the current research is 
to extend the characterization theorem \cite{KroPot:4} 
for the distance-$2$ 
MDS codes in $H(n,4)=D(0,n)$ to the 
maximum independent sets 
in $D(m,n)$ with $m>0$.
The characterization theorem 
is formulated in Section~\ref{s:res}
and proven in Section~\ref{s:proof}.
In conjunction with the results of \cite{BesKro:MDS}, this result can also be considered as completing the characterization 
of the MDS codes in the Doob graphs (there is also an infinite class of MDS codes with distance coinciding to the graph diameter
and a few distance-$3$ and distance-$4$ codes \cite{BesKro:MDS}).

As an intermediate result, in Section~\ref{s:key}
we prove a connection between properties of related objects,
\MMDS codes, 
which can be defined as the sets with largest edge boundary
(see Subsection~\ref{s:cut}). 
This result also generalizes
 its partial case $m=0$, considered in \cite{Kro:2codes}.
 
In Section~\ref{s:crc},
we consider alternative definitions 
of MDS and \MMDS codes in Doob graphs.
In particular, 
we show that
these two classes can be characterized
in terms of equitable $2$-partitions
with quotient matrices whose nontrivial eigenvalue 
is the 
minimum eigenvalue of the graph.
In $D(m,0)$, we construct an equitable partition
that has intermediate parameters between 
MDS and \MMDS codes
(strictly speaking, the quotient matrix 
of a new equitable partition is the
arithmetic average of two quotient matrices
corresponding to an MDS code and to a \MMDS code).
The existence of such ``intermediate'' objects 
between MDS and \MMDS codes
is a new effect, 
which has no analogs in Hamming graphs.

The next two sections contain preliminaries and auxiliary facts.


\section{Preliminaries}\label{s:pre}

\begin{figure}[hbt]
\begin{center}
\begin{tikzpicture}[
xscale=0.6,yscale=0.5,xslant=-0.5,
nn/.style={circle,fill=white,draw=black,thick, 
           inner sep=0.7pt}]
           \begin{scope}  
\clip [] (-0.45,-0.45) rectangle (3.45,3.45);
\draw[ystep=1,xstep=1,very thick] (-4.9,-2.1) grid (5.4,3.9);
\draw[xslant=1,ystep=9,xstep=1,very thick] (-3.4,-2.1) grid (6.4,3.9);
\end{scope}
\draw [dashed,thin] (-0.5,-0.5) rectangle (3.5,3.5);
\draw [dashed,thin,<->,rounded corners=3.5mm] (-0.5,2.7) -- (-1.2,2.7) -- (-1.2,3.8) -- (4.2,3.8) --(4.2,2.7) -- (3.5,2.7);
\draw [dashed,thin,<->,rounded corners=3.5mm] (-0.5,0.3) -- (-1.2,0.3) -- (-1.2,-0.8) -- (4.2,-0.8) --(4.2,0.3) -- (3.5,0.3);
\draw [dashed,thin,<->,rounded corners=3.5mm] (0.3,-0.5) -- (0.3,-1.2) -- (-0.8,-1.2) -- (-0.8,4.2) --(0.3,4.2) -- (0.3,3.5);
\draw [dashed,thin,<->,rounded corners=3.5mm] (2.7,-0.5) -- (2.7,-1.2) -- (3.8,-1.2) -- (3.8,4.2) --(2.7,4.2) -- (2.7,3.5);
\draw 
  (0,0) node [nn] {$\scriptstyle 00$} +(1,0) node [nn] {$\scriptstyle  10$} +(2,0) node [nn] {$\scriptstyle  20$} +(3,0) node [nn] {$\scriptstyle  30$}
++(0,1) node [nn] {$\scriptstyle  01$} +(1,0) node [nn] {$\scriptstyle  11$} +(2,0) node [nn] {$\scriptstyle  21$} +(3,0) node [nn] {$\scriptstyle  31$}
++(0,1) node [nn] {$\scriptstyle  02$} +(1,0) node [nn] {$\scriptstyle  12$} +(2,0) node [nn] {$\scriptstyle  22$} +(3,0) node [nn] {$\scriptstyle  32$}
++(0,1) node [nn] {$\scriptstyle  03$} +(1,0) node [nn] {$\scriptstyle  13$} +(2,0) node [nn] {$\scriptstyle  23$} +(3,0) node [nn] {$\scriptstyle  33$};
\end{tikzpicture}
\end{center}
\caption{The Shrikhande graph drown on a torus; the vertices are identified with the elements of $Z_4^2$}
\label{fig:Sh0123}
\end{figure}
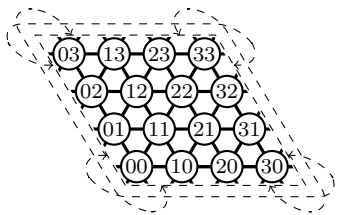
The \emph{Shrikhande graph} $\Sh$ 
is the Cayley graph of the group $Z_4^2$
with the connecting set $\{01,10,11,03,30,33\}$
(the vertices of the graph are the elements of the group $Z_4^2$, which will be denoted $00$, $01$, $02$, $03$, $10$, ..., $33$; 
two vertices are adjacent if and only if their difference belongs to the connecting set)
see Fig.~\ref{fig:Sh0123}.
The complete graph $K=K_4$ of order $4$
is the Cayley graph of the group $Z_2^2$
with the connecting set $\{01,10,11\}$.
The Cartesian product of $m$ copies of $\Sh$
and $n$ copies of $K_4$ will be denoted by $D(m,n)$.
This graph is called a \emph{Doob graph} if $m>0$,
while $D(0,n)$ is a $4$-ary \emph{Hamming graph}.
Note that $D(m,n)$ is a Cayley graph of $(Z_4^2)^m\times (Z_2^2)^n$, with the corresponding connecting set.

Given a graph $G$, by $\VV G$ we denote its set of vertices. 
Two subsets of $\VV G$ are said to be \emph{equivalent} if there is a graph automorphism that maps one subset to the other.

An independent set of vertices of maximal cardinality, 
i.e. $4^{2m+n-1}$, in $D(m,n)$ is called a distance-$2$ MDS code
(the independence number  $4^{2m+n-1}=|\VV D(m,n)|/4$ follows easily from the independence numbers $4$ and $1$ of $\Sh$ and $K$, respectively).
The two inequivalent MDS codes in $D(1,0)$ are shown in Fig.~\ref{f:mds}.

\begin{lemma}\label{l:subgraph}
  If a subgraph $G$ of $D(m,n)$ is isomorphic to a Doob graph (with smaller parameters), then the intersection of an MDS code with $\VV G$ is an MDS code in $G$. 
\end{lemma}

\begin{figure}
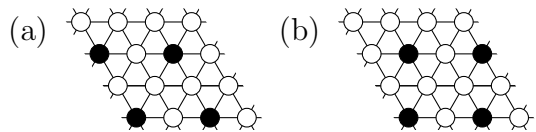

\centering
(a) \raisebox{-12mm}{\sh{nz nz nz nz}{zz nz zz nz}{nz nz nz nz}{zz nz zz nz}}\ 
(b) \raisebox{-12mm}{\sh{nz nz nz nz}{nz zz nz zz}{nz nz nz nz}{zz nz zz nz}}
\caption{All MDS codes in $\Sh$, up to isomorphism}
\label{f:mds}
\end{figure}

A function 
$f:\VV D(m,n) 
\to 
\VV K_4$ 
is called a \emph{latin coloring} 
if the preimage of every value is an MDS code; i.e., no two neighbor vertices have the same colors. In the case $m=0$, 
the latin colorings are known as the \emph{latin hypercubes} (if $n=2$, the \emph{latin squares}) of order $4$.
The case $m=1$, $n=0$ is illustrated in Fig.~\ref{f:latin}.
\begin{figure}
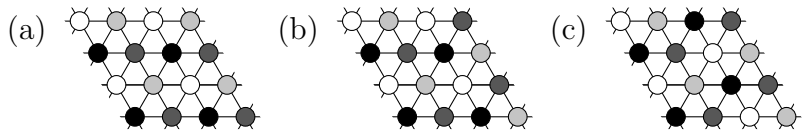

\centering
(a) \raisebox{-12mm}{\sh{nz vz nz vz}{zz wz zz wz}{nz vz nz vz}{zz wz zz wz}}\ 
(b) \raisebox{-12mm}{\sh{nz vz nz wz}{zz wz zz vz}{nz vz nz wz}{zz wz zz vz}}\ 
(c) \raisebox{-12mm}{\sh{nz vz zz wz}{zz wz nz vz}{nz vz zz wz}{zz wz nz vz}}
\caption{All latin colorings of the Shrikhande graph, 
up to isomorphism}
\label{f:latin}
\end{figure}

The \emph{graph}\footnotemark[1]%
\footnotetext[1]{We will use this footnote mark 
to separate this notion from another 
meaning of the word ``graph'' 
(a pair of a set of vertices and a set of edges).} 
$\{ (y_1,...,y_m,x_1,...,x_n,x_0) \mid 
x_0 = f(y_1,...,y_m,x_1,...,x_n) \}$
of a latin coloring $f$ is always an MDS code. Inversely, any MDS code in
$D(m,n+1)$ is a graph\footnotemark[1] of a latin coloring of $D(m,n)$
(however, the MDS codes in $D(m,0)$ cannot be represented in such a manner).
Moreover, if $f$ and $f'$ are latin colorings of 
$D(m,n)$ and $D(m',n')$, then the set
\begin{equation}\label{eq:redu}
\{ (\bar y,\bar y',\bar x, \bar x') 
\in \VV D(m'+m,n'+n)
\mid 
f(\bar y,\bar x) = f'(\bar y',\bar x') \}
\end{equation}
is an MDS code in $D(m'+m,n'+n)$.
If $2m+n>1$ and $2m'+n'>1$, then 
the MDS code (\ref{eq:redu}) is called
\emph{reducible}, as well as all codes 
obtained from it by coordinate permutation.

A set $M$ of vertices of $D(m,n)$ is called a 
\emph{\MMDS code} (two-fold MDS code) if every Shrikhande subgraph of $D(m,n)$ intersects with $M$ in $8$ vertices
that form two disjoint MDS codes in $\Sh$ (see Fig~\ref{f:2mds}) and every clique of order $4$ contains exactly $2$ elements of $M$.
The union of two disjoint MDS codes is always a \MMDS code; such a \MMDS code will be called \emph{bipartite}.
The complement $\VV D(m,n) \backslash M$ of any \MMDS code $M$ is also a \MMDS code, which will be denoted by $\overline M$.
However, it is not known if the complement of any bipartite \MMDS code is also a bipartite \MMDS code
(the known solution \cite{Pot:2012:partial} for $D(0,n)$ is far from being easy).

A connected component of a subgraph of $D(m,n)$ induced by a \MMDS code $M$ will be called a \emph{component} (of $M$).
A \MMDS code is called \emph{connected} if it is a component itself.
The union of (one or more) components of the same \MMDS code will be referred to as a \emph{multicomponent}.

A \MMDS code is called \emph{decomposable} (\emph{indecomposable}) if its characteristic function
can (cannot)  be represented as a modulo-$2$ sum of two $0,1$-functions in disjoint nonempty collections of variables.
It follows that these functions are the characteristic functions of \MMDS codes in Doob graphs, which, in their turns, 
can be decomposable or not. 
As a result, 
the characteristic function 
of any \MMDS code 
is a sum of the characteristic functions
of one, two, or more (up to $m+n$) indecomposable \MMDS codes. 
With some natural assumptions, 
such a representation is unique:

\begin{figure}
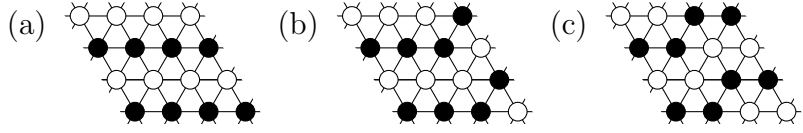

\centering
(a) \raisebox{-12mm}{\sh{nz nz nz nz}{zz zz zz zz}{nz nz nz nz}{zz zz zz zz}}\ 
(b) \raisebox{-12mm}{\sh{nz nz nz zz}{zz zz zz nz}{nz nz nz zz}{zz zz zz nz}}\ 
(c) \raisebox{-12mm}{\sh{nz nz zz zz}{zz zz nz nz}{nz nz zz zz}{zz zz nz nz}}
\caption{All \MMDS codes in the Shrikhande graph, up to isomorphism}
\label{f:2mds}
\end{figure}

\begin{lemma}\label{l:can2MDS}
The characteristic function $\chi_M$ 
of any \MMDS code $M\subset \VV D(m,n)$ 
has a unique representation in the form
\begin{equation}\label{eq:can2MDS}
\chi_M(\bar x,\bar y) \equiv 
\chi_{M_1} (\tilde x_1,\tilde y_1) + \ldots + 
\chi_{M_k}(\tilde x_k,\tilde y_k) + \sigma \bmod 2
\end{equation}
where 
\begin{itemize}
\item[$\bullet$]
$\bar x = (x_1,...,x_m) \in \Sh^m$; 

\item[$\bullet$] 
$\bar y = (y_1,...,y_n) \in K^n$; 

\item[$\bullet$] 
$\tilde x_i = (x_{j_{i,1}},...,x_{j_{i,m_i}}) \in \Sh^{m_i}$, $i=1,...,k$;

$m = \sum_{1}^{k} m_i$;

$\{1,...,m\} = \{j_{1,1},...,j_{1,m_1},\ldots,j_{k,1},...,j_{k,m_k}\}$;

\item[$\bullet$] 
$\tilde y_j = (y_{l_{i,1}},...,y_{l_{i,n_i}}) \in K^{n_i}$, $i=1,...,k$;

$n = \sum_{1}^{k} n_i$;

$\{1,...,n\} = \{l_{1,1},...,l_{1,n_1},\ldots,l_{k,1},...,l_{k,n_k}\}$;

\item[$\bullet$] 
for all $i\in\{1,...,k\}$ it holds  $m_i+n_i \ge 1$ and $M_i$ is an indecomposable \MMDS code 
not containing the all-zero tuple;
\item[$\bullet$] 
$\sigma \in\{0,1\}$.
\end{itemize}
In this notation, $M$ is indecomposable if and only if $k=1$. 
\end{lemma}
\begin{proof}
We first note that the graph structure is not essential in this lemma, because the adjacency is not used.
Next, the statement of the lemma holds for any vertex set $M$, not only for a \MMDS code.
It was proven in \cite[Lemma~A.4]{Kro:2codes}, but formally, only for the case when all the arguments take their values from the same set
(in our context, this is the case if $n=0$). 
However, this restriction is not essential for the proof, and that proof can be referred for the general case
(rewriting all arguments here is not reasonable; they are not difficult but accurate arguing takes some place).
Another way is to satisfy this restriction by some artificial trick, see the proof of Lemma~\ref{l:K16K4} below. 
\end{proof}

Obviously, an \MMDS code $M$ is decomposable if and only if $k>1$ in the representation (\ref{eq:can2MDS}).


\section[Linear 2 x MDS codes and semilinear MDS codes]{Linear \MMDS codes and semilinear MDS codes}\label{s:lin-slin}

We will say that a \MMDS code is \emph{linear} if $k=m+n$ in the representation (\ref{eq:can2MDS}), i.e.,
all $M_i$ are \MMDS codes in $\Sh$ or $K$, and, moreover, the ones that are in $\Sh$ are not connected (Fig. \ref{f:2mds}(a)). 
It is easy to see that all linear \MMDS codes in a given $D(m,n)$ are equivalent.

We will say that an MDS code is \emph{semilinear} if it is a subset of a linear \MMDS code.
In the rest of this section, after auxiliary statements, we evaluate the number of semilinear MDS codes.
We omit the proof of the next lemma as it is obvious.
\begin{lemma}\label{l:No-comp}
Let a characteristic function $\chi_M$ of a decomposable \MMDS code $M$ has a decomposition
$$ 
\chi_M(\bar x,\bar y) \equiv 
\chi_{M_1} (\tilde x_1,\tilde y_1) + 
\chi_{M_2}(\tilde x_2,\tilde y_2) + \sigma \bmod 2
$$ 
where each $M_i$, $i=1,2$, as well as its complement, has $N_i$ components.
Then $M$, as well as its complement, has $2 N_1 N_2$ components.
\end{lemma}

\begin{lemma}\label{l:2mds-mds}
Let $M$ be a \MMDS code in $D(m,n)$, and let $N$ be the number of components in $M$.

a) If $M$ is a linear \MMDS code then $N=2^{2m+n-1}$.

b) If $M$ is bipartite, then the number of MDS codes included in $M$ is $2^N$.
\end{lemma}
\begin{proof}
a) The statement follows by induction from Lemma~\ref{l:No-comp}.

b)  As follows from the hypothesis, the subgraph $D_M$ of $D(m,n)$ induced by $M$ is bipartite. 
By the definition of an MDS code, a part of $D_M$ is an MDS code. 
To choose a part of $D_M$, one should independently choose one of two parts of every connected components of $D_M$. 
So, the number of ways is $2^N$.
\end{proof}

\begin{lemma}\label{l:No-lin}
 The number of linear \MMDS codes in $D(m,n)$ is $2\cdot 3^{m+n}$.
\end{lemma}
\begin{proof}
Every linear \MMDS codes can be represented as (\ref{eq:can2MDS}) where for each $i\in\{1,2,...,k=m+n\}$ the set 
$M_i$ is a \MMDS code in $\Sh$ equivalent to the one in Fig.~\ref{f:2mds}(a) or a \MMDS code in $K$ and, moreover, 
$00\not\in M_i$.
In any case, $M_i$ can be chosen in $3$ ways. 
Since $\sigma$ can be chosen in $2$ ways,
we have totally $2\cdot 3^{m+n}$ ways to specify a linear \MMDS code.
\end{proof}

\begin{lemma}\label{l:No-semi}
 The number of semilinear  MDS codes in $D(m,n)$  is 
 $$2\cdot 3^{m+n}\cdot 2^{2^{2m+n-1}}(1+o(1)) \quad\mbox{ as } m+n \to \infty.$$
\end{lemma}
\begin{proof} There are $2\cdot 3^{m+n}$ linear \MMDS codes (Lemma~\ref{l:No-lin}), 
each including $2^{2^{2m+n-1}}$ semilinear MDS codes (Lemma~\ref{l:2mds-mds}).
 It remains to understand that almost every semilinear MDS code is included in only one linear \MMDS code.
 One of simple explanations of this fact is that two different linear \MMDS codes $M'$, $M''$ 
 can include at most one common MDS code. 
 Indeed each of  $M'$, $M''$ is a coset of a subgroup of $(Z_4^2)^m\times (Z_2^2)^n$, 
 so the cardinality of their intersection cannot be larger than $|M'|/2=|M''|/2=4^{2m+n-1}$,
 which is the cardinality of an MDS code. 
\end{proof}

\section{Main results}\label{s:res}
We are now ready to formulate the main result of the paper, 
which will be proven in the next two sections.  

\begin{theorem}\label{th:main}
Every MDS code in $D(m,n)$ is semilinear or reducible.
\end{theorem}
Note that `or' is not `xor' here; a reducible MDS code (\ref{eq:redu}) can also be semilinear.
This happens when both graphs\footnotemark[1] of $f$ and $f'$ are semilinear and, 
moreover, the representations (\ref{eq:can2MDS}) of the corresponding linear \MMDS codes
have the same summand $\chi_{M^*}(x_0)$ corresponding to the dependent variable $x_0$
($M^*$ can be $\{(0,1),(1,0)\}$, $\{(0,1),(1,1)\}$, or $\{(1,0),(1,1)\}$, 
but it is the same for both $f$ and $f'$).

\begin{corollary}\label{cor:N}
The number of MDS codes in $D(m,n)$ has the form 
$$2\cdot 3^{m+n}\cdot 2^{2^{2m+n-1}}(1+o(1)) \qquad\mbox{as\quad$m+n \to \infty$.}$$ 
\end{corollary}
\begin{proof}
By Lemma~\ref{l:No-semi}, the number of semilinear MDS codes is $2\cdot 3^{m+n}\cdot 2^{2^{2m+n-1}}(1+o(1))$.
 The number of reducible MDS codes is asymptotically inessential comparing with this value 
 (the arguments for the general case do not differ from those for $D(0,n)$, considered in \cite{PotKro:asymp}).
\end{proof}
Corollary~\ref{cor:N} improves the result of \cite{Kro:2015:N-MDS-Doob}, where simpler arguments are used to establish the asymptotic of the logarithm of the number of MDS codes in $D(m,n)$.

\section{A key proposition}\label{s:key}
In this section, we will prove the following proposition, which will be used in the proof of the main theorem.

\begin{proposition}\label{p:key}
  A \MMDS code in $D(m,n)$, $(m,n)\ne (1,0)$, is decomposable if and only if it is not connected.
\end{proposition}

The following fact is easy to check directly.
\begin{lemma}\label{l:2sh}
  Let $M$ and $M'$ be two \MMDS codes in $\Sh$. 
  Then,  either $M$ and $M'$ coincide, 
  or they are are disjoint (i.e., $M\cup M' = \VV \Sh$), 
  or every component of $M$ intersects with every component of $M'$.
\end{lemma}

The following partial case of Proposition~\ref{p:key} will be used as an auxiliary statement.
\begin{lemma}\label{l:2mds2}
  A \MMDS code $M$ in $D(m,n)$, $m+n=2$, is decomposable if and only if it is not connected.
\end{lemma}
\begin{proof}
By Lemma~\ref{l:No-comp} every decomposable \MMDS code is not connected, which is the 'only if' part of the statement.
For the `if' part, we have to show that any \MMDS code $M$ in $D(m,n)$, $m+n=2$, is decomposable or connected.

The case of $D(0,2)=K\times K$ is trivial, as there are only two nonequivalent \MMDS codes, one is decomposable, the other is connected.

The case of $D(1,1)=\Sh\times K$ is also simple. 
Denote $M_y = \{ x \in \VV \Sh \mid  (x,y) \in M \}$ for each $y\in \VV K$.
Then $M_y$ is a \MMDS code in $\Sh$. 
If for all $y'$, $y''$ from $\VV K$ 
the sets $M_{y'}$ and $M_{y''}$ 
are either coinciding or disjoint,
then, readily, $M$ is decomposable.
Assume that for some  $y'$, $y''\in\VV K$ the sets $M_{y'}$ and $M_{y''}$ are neither coinciding nor disjoint. 
By Lemma~\ref{l:2sh}, the corresponding $16$ vertices of $M$ belong to the same component.
Moreover, for every $y$ from $\VV K$, the set $M_y$ is neither coinciding nor disjoint with at least one of $M_{y'}$, $M_{y''}$.
It follows that $M$ is connected.

It remains to consider the case of $D(2,0)=\Sh\times\Sh$.
For $a\in \Sh$, denote 
\begin{eqnarray*}
R_a=\{x_1\in \Sh \mid (a,x_1) \in M\} &\mbox{ and }& aR_a=\{(a,x_1) \in M\}, \\
L_a=\{x_1\in \Sh \mid (x_1,a) \in M\} &\mbox{ and }& L_a a=\{(x_1,a) \in M\}.
\end{eqnarray*}
Consider the case when $L_{00}$ is equivalent to the \MMDS code in Fig.~\ref{f:2mds}(a)
(the other two cases are considered similarly);
without loss of generality, assume $L_{00}=\{00,01,02,03,20,21,22,23\}$.
Now consider the component $\{00,\linebreak[2]01,\linebreak[2]02,\linebreak[2]03\}$ of $L_{00}$.

(i)
If $R_{00}=R_{01}=R_{02}=R_{03}$, 
then for every $a$ from $\Sh$, the set $\{00,\linebreak[2]01,\linebreak[2]02,\linebreak[2]03\}$ is a component of $L_a$ or $\overline L_a$; 
hence, by Lemma~\ref{l:2sh}, $L_a=L_{00}$ or $\overline L_a=L_{00}$. 
In this case, readily, $M$ is decomposable.

(ii)
So, without loss of generality we can assume that 
$R_{00}$ and $R_{01}$ are different.
Since they intersect in $(00)$
(indeed, $(00,00),(01,00)\in L_{00}\subset M$), 
Lemma~\ref{l:2sh} means that 
$00R_{00}$ and $01R_{01}$ 
lie in a one component of $M$.
Next, $R_{02}$ intersects with $R_{00}$;
hence, $02R_{02}$ lies in the same component as $00R_{00}$.
Similarly, $03R_{03}$.

By similar arguments, each of $S_0=00R_{00}\cup 01R_{01}\cup 02R_{02}\cup 03R_{03}$, 
$S_1=10R_{10}\cup 11R_{11}\cup 12R_{12}\cup 13R_{13}$, $S_2=20R_{20}\cup 21R_{21}\cup 22R_{22}\cup 23R_{23}$, 
$S_3=30R_{30}\cup 31R_{31}\cup 32R_{32}\cup 33R_{33}$ is a subset of a component of $M$,
and it remains to show that this component is common for all the four sets.
For example, let us show that it is common for $S_0$ and
$S_1$. Indeed, since $R_{00}$ and $R_{01}$ are different,
$R_{11}$ intersects with at least one of them; this means that there is an edge connecting 
$00R_{00}\cup 01R_{01}$ and $11R_{11}$. Similarly, there are edges connecting $S_0$ and $S_3$.
The same argument is applicable to $S_2$ and $S_3$, 
except for the case $R_{20}=R_{21}=R_{22}=R_{23}$.
Excluding this case by the argument similar to (i),
we find that $M$ consists of one component.
\end{proof}

\begin{lemma}\label{l:nei}
  Let $M$ be a \MMDS code in $D(m,n)$, let $L$ be a multicomponent of $M$, and let $L_1$ be the set of vertices of $D(m,n)$ 
  at distance $1$ from $L$. Then $L_1$ is a multicomponent of $\overline M$.
\end{lemma}
\begin{proof}
Clearly, $L_1$ is a subset of $\overline M$. 
It remains to show that
for each $b\in L_1$ every $c\in\overline M$ adjacent to $b$ also belongs to $L_1$.
By the definition of $L_1$, there is $a\in L$ adjacent to $b$.
Assume that $a$ and $b$ differ in the $i$th coordinate, 
while $b$ and $c$ differ in the $j$th coordinate.

(*) We will show that {\it there is $d\in L$ adjacent to $c$ and differing from $c$ in the $i$th coordinate}.
If $i=j$ is a $K$-coordinate, 
 then this claim is obvious (we can take $d=a$).
If $i=j$ is a $\Sh$-coordinate, 
 then it is also easy to see.
Assume $i \ne j$ and consider the subgraph $D$ 
(isomorphic to  $D(0,2)$, $D(1,1)$, or $D(2,0)$) corresponding 
to these two coordinates and containing the vertices $a$, $b$, and $c$.
Let $M'$ be the intersection of $M$ with this subgraph.
We know that $M'$ is a \MMDS code in $D$ and hence, 
by Lemma~\ref{l:2mds2}, it is connected or decomposable.
In the first case, $M' \subseteq L$ and (*) is trivial.
In the last case, the vertex $d$ coinciding with $c$ in all coordinates 
except the $i$th one and coinciding with $a$ in the $i$th coordinate
must belong to $L$ (indeed, from the decomposability we have 
\begin{eqnarray*}
\chi_{M'}(a) + \chi_{M'}(b) + \chi_{M'}(c) + \chi_{M'}(d) \qquad\qquad\qquad\qquad\qquad&&\\
 {}
 \stackrel{(\ref{eq:M1M0})}{=}(\chi_{M'_1}(a_i)+\chi_{M'_2}(a_j) +\sigma) 
+(\chi_{M'_1}(c_i)+\chi_{M'_2}(a_j) +\sigma) &&\\
{}+(\chi_{M'_1}(c_i)+\chi_{M'_2}(c_j) +\sigma)
+(\chi_{M'_1}(a_i)+\chi_{M'_2}(c_j) +\sigma)&= &0\bmod 2,
\end{eqnarray*}
hence $d\in M'$; since $d$ and $a$ are adjacent, we also have $d\in L$).

So, (*) holds, and $L_1$ consists of components of $\overline M$.
\end{proof}

\begin{corollary}\label{c:nei}
In $D(m,n)$, an \MMDS code is uniquely determined 
by any of its multicomponents. 
In other words,
two \MMDS codes with a common multicomponent coincide.
\end{corollary}
\begin{proof}
 Let $L_0$ be a multicomponent of an \MMDS code $M$. 
 Define recursively $L_{i+1}$ 
 as the set of vertices of $D(m,n)$
 at distance $1$ from $L_i$. 
 Since the diameter of the graph is $2m+n$ 
 and noting that 
 $L_{i}\subseteq L_{i+2}$,
 we find 
 $$ \VV D(m,n) = \bigcup_{i=0}^{2m+n} L_{i} = L_{2m+n-1}\cup L_{2m+n}. $$
 By Lemma~\ref{l:nei}, $L_1,L_3,L_5,\ldots \subseteq \overline M$ and
 $L_0,L_2,L_4,\ldots \subseteq  M$.
 We conclude that $M = L_{2m+n-1}$ or $M = L_{2m+n}$, depending on the parity of $2m+n$.
\end{proof}

\begin{lemma}\label{l:Sh-to-K}
  Let $M$ be a \MMDS code in $D(m,n)$ where $m>0$ and $m+n>1$. 
  Let, for some $a_2,...,a_m\in \VV \Sh$ and $b_1,...,b_n\in \VV K$,  
  $$M' = \{(x_1,a_2,...,a_m,b_1,...,b_n)\in M \mid x_1 \in \VV\Sh \}.$$ 
  Then either all $8$ elements of $M'$ belong to the same component of $M$ or 
  $M$ is decomposable and 
  \begin{eqnarray}\label{eq:M1M0}
     \chi_{M}(x_1,x_2,...,x_m,y_1,...,y_n)\qquad\qquad\qquad\qquad\qquad\qquad\qquad\qquad\qquad  \nonumber\\
 {}  = \chi_{M'}(x_1,a_2,...,a_m,b_1,...,b_n)
   + \chi_{M_0}(x_2,...,x_m,y_1,...,y_n) +1 \bmod 2
  \end{eqnarray}
  for some \MMDS code $M_0$.
\end{lemma}
\begin{proof} 
  For  $\alpha = (\alpha_2,...,\alpha_m)\in \VV \Sh^{m-1}$ 
  and $(\beta_1,...,\beta_n)\in \VV K^n$, 
  denote
   $M_{\alpha,\beta} = \{ x_1 \in \VV\Sh \mid(x_1,\alpha_2,...,\alpha_m,\beta_1,...,\beta_n)\in M \}$.
   Clearly, $M_{\alpha,\beta}$ is a \MMDS code in $\Sh$. 
   If $M_{a,b}$ is connected, where $a=(a_2,...,a_m)$, $b=(b_1,...,b_n)$,
   then, trivially, $M'$ lies in a component of $M$ and
   the statement holds.
   
   Assume that $M_{a,b}$ is not connected and, moreover, 
   $M'$ intersects with two components of $M$.
   Without loss of generality, 
   suppose $(00,a_2,...,a_m,b_1,...,b_n)\in M$, 
   i.e., $00 \in M_{a,b}$.   
   Define the \MMDS code $M_0=\{(\alpha,\beta) \mid (00,\alpha,\beta)\in M \}$; 
   in particular, $(a,b)\in M_0$.
   Our goal is now to prove  (\ref{eq:M1M0}), or, equivalently,
   to prove that $M$ coincides with the \MMDS code $T\subset \VV D(m,n):$
  \begin{eqnarray}\label{eq:M1M0var}
     &&\chi_{T}(x_1,x_2,...,x_m,y_1,...,y_n) \nonumber \\
     &&{}\quad = \chi_{M_{a,b}}(x_1) + \chi_{M_0}(x_2,...,x_m,y_1,...,y_n) +1 \bmod 2.
  \end{eqnarray}
   Let $L_0$ be the component of $M_0$ that contains $(a,b)$.
   For every $(a',b')\in M_0$ adjacent to $(a,b)$ 
   we have $M_{a',b'}=M_{a,b}$ 
   (otherwise, by Lemma~\ref{l:2sh}, 
   $M'$ lies in one component, 
   contradicting our assumption).
   Similarly,  $M_{a',b'}=M_{a,b}$ 
   for every $(a',b')$ from $L_0$.
   Now we see that the set 
  \begin{eqnarray*}
     L=\{(x_1,x_2,...,x_m,y_1,...,y_n)\,:\, x_1\in M_{a,b},\ (x_2,...,x_m,y_1,...,y_n)\in L_0\},
  \end{eqnarray*}
  is a multicomponent of $T$ (\ref{eq:M1M0var})
  and, moreover, 
  is a subset of $M$.
  By Corollary~\ref{c:nei}, we have $M=T$.  
\end{proof}

\begin{lemma}[{\cite[Corollary~4.2, Remark~4.3]{Kro:2codes}}]\label{l:Kq}
  Let $G=K_q^n$ be the Cartesian product of $n$ copies of the complete graph $K_q$ of even order $q$.
  Let $M$ be a set of vertices of $G$ such that every clique of order $q$ contains exactly $q/2$ elements of $M$.
  Then the subgraph $G_{M}$ of $G$ induced by $M$ is disconnected if and only if the characteristic function $\chi_M$ of $M$ is decomposable
  into the sum $\chi_M(z) = \chi_{M'}(z')+ \chi_{M''}(z'') \bmod 2$, where $M'\subset \VV K_q^{n'}$, $M''\subset \VV K_q^{n''}$, 
  and $z'$, $z''$ are nonempty disjoint collections of variables from $z=(z_1,...,z_n)$ of length $n'$ and $n''$ respectively, $n'+n''=n$.
\end{lemma}

\begin{lemma}\label{l:K16K4}
  Let $G^*=G^*(m,n)=K_{16}^m\times K_4^n$ be the Cartesian product of $m$ copies of the complete graph $K_{16}$ of order $16$
  and $n$ copies of the complete graph $K_{4}$ of order $4$.
  Let $M^*$ be a set of vertices of $G^*$ such that every clique $K$ maximal by inclusion (it follows that $|K|=4$ or $|K|=16$) contains exactly $|K|/2$ elements of $M^*$.
  Then the subgraph $G^*_{M^*}$ of $G^*$ induced by $M^*$ is disconnected if and only if the characteristic function $\chi_{M^*}$ of $M^*$ is decomposable
  into the sum $\chi_{M^*}(z) = \chi_{M'}(z')+ \chi_{M''}(z'') \bmod 2$, where $M'\subset \VV G^*(m',n')$, $M''\subset \VV G^*(m'',n'')$, 
  and $z'$, $z''$ are nonempty disjoint collections of variables from $z=(x_1,...,x_m,y_1,...,y_n)$.
\end{lemma}
\begin{proof}
We map each vertex $(x_1,...,x_m,y_1,...,y_n)$ of $G^*$  to $4^n$ elements 
$$(x_1,...,x_m,(y_1,z_1),...,(y_n,z_n)), \quad z_i\in\{0,1,2,3\}, \quad i=1,...,n,$$ 
which will be treated as vertices of $K_{16}^{m+n}$.
Then $M^*$ is mapped into a set $M$ of vertices of $K_{16}^{m+n}$. 
It is easy to see that $M$ satisfies the hypothesis of Lemma~\ref{l:Kq}. 
Moreover, the subgraph $G_M$ is connected if and only if the subgraph $G^*_{M^*}$ is connected;
and the characteristic function $\chi_{M}$ is decomposable if and only if $\chi_{M^*}$ is decomposable.
Then, the statement follows from Lemma~\ref{l:Kq}. 
\end{proof}

\begin{proof}[of Proposition~\ref{p:key}]
Assume that $M$ is a \MMDS code in $D(m,n)$, where $(m,n)\ne(1,0)$. 
If $M$ is connected, then it is indecomposable by Lemma~\ref{l:No-comp}.

Assume that $M$ induces a disconnected subgraph of $D(m,n)$. 
Since  $K_{16}^m \times K_4^n$ is obtained from $D(m,n)$ by adding some edges,
the subgraph of $K_{16}^m \times K_4^n$ induced by $M$ can be connected or not.

If $M$ induces a disconnected subgraph of $K_{16}^m \times K_4^n$,
then the statement follows from Lemma~\ref{l:K16K4}.

It remains to consider the case when the subgraph of $K_{16}^m \times K_4^n$ induced by $M$ is connected.
This means that some Shrikhande subgraph of $D(m,n)$ intersects with two components of $M$
(then, adding the absent edges in this subgraph connects these two components), in particular, $m\ge 1$.
But then $M$ is decomposable by Lemma~\ref{l:Sh-to-K}.
This proves the proposition.
(However, to make the situation clear, we can further note that the last case is contradictory:
the decomposability of $M$ implies that the induced subgraph of $K_{16}^m \times K_4^n$ is not connected.)
\end{proof}

\section{Proof of the main theorem}\label{s:proof}
\begin{proof}[of Theorem~\ref{th:main}]
In the case $m=0$, 
the statement coincides with the main theorem of \cite{KroPot:4}. 
The case $(m,n)=(1,0)$ is trivial.
Consider an MDS code $C$ in $D(m,n)$, 
$m>0$, $m+n>1$.
  For  $\alpha = (\alpha_2,...,\alpha_m)\in \VV \Sh^{m-1}$ 
  and $\beta=(\beta_1,...,\beta_n)\in \VV K^n$, 
  denote
   $C_{\alpha,\beta} = \{ x_1 \in \VV\Sh \mid(x_1,\alpha_2,...,\alpha_m,\beta_1,...,\beta_n)\in C \}$.

Assume that for all $\alpha$ and $\beta$ 
the MDS code $C_{\alpha,\beta}$ 
is equivalent to the code shown 
in Fig.~\ref{f:mds}(a), 
i.e., one of 
$C_{00}=\{00,02,20,22\}$, 
$C_{01}=\{01,03,21,23\}$, 
$C_{10}=\{10,12,30,32\}$, 
$C_{11}=\{11,13,31,33\}$.
Then, it is easy to see that $C$ is reducible:
$$ C = \{ (y,\bar y',\bar x') \in D(m,n) \mid f(y)=f'(y',x') \} $$
for some 
$f':\VV D(m-1,n) \to \{00,\linebreak[2]01,\linebreak[2]10,\linebreak[2]11\}$ 
and for 
$f:\VV D(1,0) \to \{00,\linebreak[2]01,\linebreak[2]10,\linebreak[2]11\}$ 
satisfying $f(y)=ij$ 
for every $y\in C_{ij}$.

Otherwise, 
we may assume without loss of generality that 
$C_{a,b} = \{00,\linebreak[2]02,\linebreak[2]21,\linebreak[2]23 \}$ 
(Fig.~\ref{f:mds}(b))
for some $a$ and $b$.
To utilize Proposition~\ref{p:key}, 
we consider the \MMDS code
$M = C \cup (C+ (01,00,...,00))$. 
We will see that it is decomposable.
Define the \MMDS code 
$M_0 = \{ (x_2,...,x_n,y_1,...,y_n) \mid (00,x_2,...,x_n,y_1,...,y_n)\in C \mbox{ or } (01,x_2,...,x_n,y_1,...,y_n)\in C \}$.
In particular, $(a,b)\in M_0$.
Similarly to $C_{\alpha,\beta}$, 
we denote
$$M_{\alpha,\beta} = \{ x_1 \in \Sh \mid(x_1,\alpha_2,...,\alpha_m,\beta_1,...,\beta_n)\in M \}.$$
In particular, 
$M_{a,b} = \{00,01,02,03,20,21,22,23\}$ 
(Fig.~\ref{f:2mds}(a)).

For every $(a',b')\in M_0$ at distance $1$ from $(a,b)$, 
the MDS code  $C_{a',b'}$ contains $01$ and, moreover, is disjoint with $C_{a,b} = \{00,02,21,23 \}$.
There is only one such code, 
$C_{a',b'} = \{01,03,20,22 \}$. 
It follows that 
$$M_{a',b'} = \{00,01,02,03,20,21,22,23\}=M_{a,b}.$$
Similarly, $M_{\alpha,\beta} = M_{a,b}$ 
for all $(\alpha,\beta)$ 
from the same connected component $L_0$ of $M_0$ as $(a,b)$.
Now, we can apply the same argument as in the proof of Lemma~\ref{l:Sh-to-K}:
define 
$$
T:\quad\chi_T = \chi_{M_{a,b}}(x_1) + \chi_{M_0(x_2,...,x_m,y_1,...,y_n)}+1 \bmod 2
$$
and 
$$
L=\{(x_1,x_2,...,x_m,y_1,...,y_n)\,:\, x_1\in M_{a,b},\ (x_2,...,x_m,y_1,...,y_n)\in L_0\};
$$
note that $L$ is a multicomponent of $T$ and a subset of $M$,
and hence $M=T$ by Corollary~\ref{c:nei}.

So, we have that $C$ is included in a decomposable \MMDS code $M$. 
Consider the decomposition (\ref{eq:can2MDS}) of $M$ into
indecomposable \MMDS codes $M_i$, $i=1,...,k$:
$$
\chi_M(\bar x,\bar y) \equiv 
\chi_{M_1} (\tilde x_1,\tilde y_1) + \ldots + 
\chi_{M_k}(\tilde x_k,\tilde y_k) + \sigma \bmod 2,\qquad  k\ge 2.
$$
By Proposition~\ref{p:key}, all $M_i$ are connected, except some \MMDS codes in $\Sh$,
which are equivalent to the code in Fig~\ref{f:2mds}(a).
Assume without loss of generality that the first $m'$ codes $M_1$, ..., $M_{m'}$ are not connected
and that they correspond to the first $m'$ variables $x_1$, ..., $x_{m'}$.
Since $M$ includes at least one MDS code $C$, all $M_i$, $i=1,...,k$ are bipartite, as well as their complements.
It follows that for each $i$ there is a latin coloring $f_i:\VV D(m_i,n_i)\to \VV K$ such that
$M_i = \{ z\in \VV D(m_i,n_i) \mid f_i(z)\in\{10,11\} \}$.
Define the linear \MMDS code $I$ by
\begin{eqnarray*}
\chi_I(x_1,...,x_{m'},z_{1},...,z_{k-m'}) &= &
\chi_{M_1}(x_1)+\ldots +\chi_{M_{m'}}(x_{m'})
\\ &&\mbox{}+
\chi_{\{10,11\}}(z_1)+\ldots +\chi_{\{10,11\}}(z_{k-m'})+
\sigma.
\end{eqnarray*}
It is straightforward that
$$ M = \{ (\bar x,\bar y)\in \VV D(m,n) \mid 
(x_1,...,x_{m'},f_{m'+1}(\tilde x_{m'+1},\tilde y_{m'+1}),...,f_{k}(\tilde x_{k},\tilde y_{k})) \in I \}.$$
Then, for every MDS code $S \subset I$, the MDS code 
\begin{equation}\label{eq:S-to-C}
 \{ (\bar x,\bar y)\in \VV D(m,n) \mid 
(x_1,...,x_{m'},f_{m'+1}(\tilde x_{m'+1},\tilde y_{m'+1}),...,f_{k}(\tilde x_{k},\tilde y_{k})) \in S \}
\end{equation}
is a subset of $M$. 
Since $I$ and $M$ consist of the same number of components, 
they include the same number of MDS codes.
In particular, $C$ is also representable 
in the form (\ref{eq:S-to-C}) for some $S$.
If $m'=m$ and $k=m+n$, then $C$ is semilinear. 
Otherwise, $k<m+n$ or $k=m+n$ and $m'<m$, and $C$ is reducible.
\end{proof}
\begin{remark}
Theorem~\ref{th:main} gives a recursive constructive description of MDS codes in $D(m,n)$. 
The class of semilinear MDS codes is constructive directly; 
the reducibility reduces the situation to the graphs of smaller diameter.
In the current paper, we prove the case $m>0$, 
but it is important to understand that the case $m=0$ is necessary for the description 
because some factors in the decomposition can have only $K$-coordinates.
It should also be noted that the approach used in this paper cannot be directly adopted to the case $m=0$.
Existence of at least one Shrikhande coordinate is essential for our current proof.
We use it to embed the MDS code into an \MMDS code having at least four connected components.
In that worst case, 
the original code is reduced to a code 
with smaller number of Shrikhande coordinates, 
the total number of coordinates remaining the same.
If $m=0$, we still can apply the same technique to a $K$-coordinate, 
but in the worst case, the \MMDS code would have only two connected components;
the induced decomposition of the embedded MDS code would be trivial and 
would not imply the reducibility.
The difficulties hidden in this case are solved 
in the papers \cite{PotKro:asymp,Kro:2codes,Kro:n-3,KroPot:4},
containing the complete proof of Theorem~\ref{th:main} for $m=0$.
\end{remark}

\section[MDS codes and 2 x MDS codes as sets with extremal properties; 1.5 x MDS codes]{MDS codes and \MMDS codes as sets with extremal properties; \MMMDS codes}\label{s:crc}

We already know that the MDS codes in a Doob graph are exactly
the maximum (by cardinality) independent sets 
(actually, we take this property as the definition of the MDS codes).
In this section, we show that the \MMDS codes also meet some extremal property.
Namely, they are exactly the sets with maximum edge boundary (cut).
Moreover, the MDS codes and the \MMDS codes
define equitable $2$-partitions 
of the Doob graph with minimum eigenvalue.

\subsection{The maximum cut}\label{s:cut}
The \emph{edge boundary} (also known as \emph{cut}) of a set $V$ of vertices of a graph $G=(\VV G, \EE G)$
is the set of edges $\{ \{v,w\}\in \EE G \mid v\in V,\,w\not\in V\}$.

\begin{proposition}\label{p:eb}
 (a) The maximum size of the edge boundary of a vertex set in $D(m,n)$ is 
 $(2m+n)4^{2m+n}$. (b) A vertex set $M$ has the maximum edge boundary if and only if
 it is a \MMDS code.
\end{proposition}

\begin{proof}
It is straightforward from the definition that a \MMDS code has the edge boundary
of size $(2m+n)4^{2m+n}$. 
It remains to show the upper bound for the statement a) 
and the `only if' part of b).

(a) $(2m+n)4^{2m+n}$ is $2/3$ of the total amount of edges in $D(m,n)$.
Since $D(m,n)$ is the Cartesian product 
of several copies of the Shrikhande graph $\Sh$ 
and several copies of the complete graph $K$ of order $4$, it is sufficient 
to show that the number of boundary edges in each of these two graphs cannot be
larger than $2/3$ of the total amount of edges, i.e., 
$\frac{2}{3}\cdot 48 = 32$ in the case of $\Sh$ 
and $\frac{2}{3}\cdot 6 = 4$ for $K$. For $K$, this is trivial.
For $\Sh$, this follows from the fact that each triangle 
(complete subgraph of order $3$) cannot have more than $2$ boundary edges,
while every edge belong to the same number of triangles.

(b) We first note that 
it is sufficient to prove the statement for the graphs $\Sh$ and $K$.
Indeed, if we have a set $M$ whose edge boundary size 
is $2/3$ of the total amount of edges in $D(m,n)$,
then the same proportion takes place in every Shrikhande or $K$ subgraph
(otherwise, there is a contradiction with p.(a) for such subgraphs).
If we already have the statement for $\Sh$ and $K$, 
then $M$ is a \MMDS code by the definition.

For $K$, it is trivial. 
Let us consider the Shrikhande graph and a vertex set $M$ with
edge bound of size $32$. We color the vertices of $M$ by black and the other 
vertices by gray. From the previous paragraph we know that each triangle has
at least one black and one gray vertex; 
we will always keep this fact in mind 
in the rest of the proof.
The Shrikhande graph has $12$ induced
(i.e., without chords) $4$-cycles. We consider three cases.

(4:0) There is an induced cycle colored into one color. 
Then, the colors of the other vertices are uniquely reconstructed
(see Fig~\ref{f:4-0}),
leading to the situation shown in Fig.~\ref{f:2mds}(a).
\begin{figure}
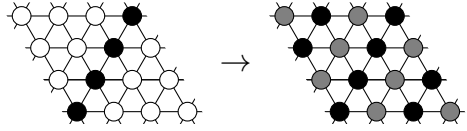

\centering
\raisebox{-7mm}{\sh{nz nz nz zz}{nz nz zz nz}{nz zz nz nz}{zz nz nz nz}}\ 
$ \rightarrow $ \raisebox{-7mm}{\sh{xz zz xz zz}{zz xz zz xz}{xz zz xz zz}{zz xz zz xz}}
\caption{There is a cycle of type 4:0}
\label{f:4-0}
\end{figure}

(3:1) There is a induced cycle with only one black vertex or only one grey vertex. 
The colors of six other vertices are uniquely reconstructed.
Choosing the color of any of the six remaining vertices uniquely leads
to one of two mutually symmetric colorings, see Fig~\ref{f:3-1}.
The result corresponds to Fig.~\ref{f:2mds}(b). 

\begin{figure}
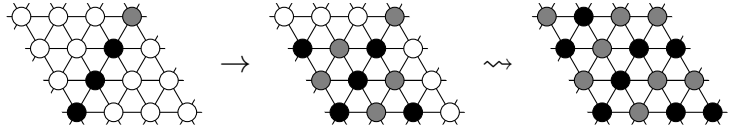

\centering
\raisebox{-7mm}{\sh{nz nz nz xz}{nz nz zz nz}{nz zz nz nz}{zz nz nz nz}}\ 
$ \!\!\!\rightarrow \!\!\!$ \raisebox{-7mm}{\sh{nz nz nz xz}{zz xz zz nz}{xz zz xz nz}{zz xz zz nz}}\ 
$\!\!\! \rightarrow \!\!\!$ \raisebox{-7mm}{\sh{xz zz xz xz}{zz xz zz zz}{xz zz xz xz}{zz xz zz zz}}%
$\!\!\!$or$\!\!\!$\raisebox{-7mm}{\sh{zz xz zz xz}{zz xz zz xz}{xz zz xz zz}{zz xz zz xz}}
\caption{There is a cycle of type 3:1}
\label{f:3-1}
\end{figure}

(2:2) All induced cycles have two black and two gray vertices.
Then, there is an induced cycle containing two neighbor black vertices,
see Fig~\ref{f:2-2}. Four vertices out of this cycle are uniquely colored.
Choosing the color of any of the eight other vertices and using
only the hypothesis of the case (2:2), we uniquely color the
seven remaining vertices, see Fig~\ref{f:2-2}. 
Both $M$ and its complement are equivalent to the set shown in Fig.~\ref{f:2mds}(c).
\begin{figure}
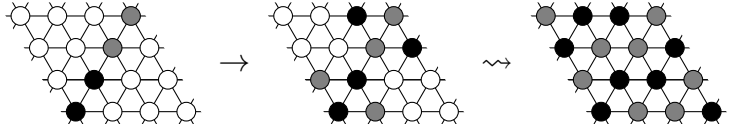

\centering
\raisebox{-7mm}{\sh{nz nz nz xz}{nz nz xz nz}{nz zz nz nz}{zz nz nz nz}}\ 
$ \!\!\!\rightarrow\!\!\! $ \raisebox{-7mm}{\sh{nz nz zz xz}{nz nz xz zz}{xz zz nz nz}{zz xz nz nz}}\ 
$ \!\!\!\rightarrow\!\!\! $ \raisebox{-7mm}{\sh{xz zz zz xz}{zz xz xz zz}{xz zz zz xz}{zz xz xz zz}}%
$\!\!\!$or$\!\!\!$\raisebox{-7mm}{\sh{zz xz zz xz}{xz zz xz zz}{xz zz xz zz}{zz xz zz xz}}\
\caption{All cycles are of type 2:2}
\label{f:2-2}
\end{figure} 
\end{proof}

\subsection{Equitable partitions}\label{s:ep}
A partition $(P_1,...,P_r)$ of the vertex set of a graph into $r$ nonempty cells 
is said to be an \emph{equitable partition}, see e.g. \cite[9.3]{GoRo}
(regular partition \cite[11.1B]{Brouwer}, 
partition design \cite{CCD:92}, 
perfect coloring \cite{FDF:PerfCol}), 
if there is an $r\times r$ matrix $(s_{ij})_{i,j=1}^r$
(the \emph{quotient matrix}) 
such that for every $i$ and $j$ from $1$ to $r$ 
every vertex from $P_i$ 
has exactly $s_{i,j}$ neighbors from $P_j$.
It is known that each eigenvalue of the quotient matrix
is an eigenvalue of the graph \cite{Brouwer}
(in a natural way, an eigenvector of the quotient matrix generates an eigenfunction of the graph).
If the graph is regular, 
then its degree is always an eigenvalue of the quotient matrix.

The graph $D(m,n)$ has the $d+1$ eigenvalues $-d$, $-d+4$, \ldots, $3d$, where $d=2m+n$ is the graph diameter.
In $\VV D(m,n)$, we consider an equitable $2$-partition $(C, \overline C)$
such that the quotient matrix has two eigenvalues $-d$ and $3d$, 
the smallest and the largest eigenvalues of $D(m,n)$. 
This matrix has the form 
$\left[ {a \atop a+d}\ {3d-a \atop 2d-a}  \right]$ for some $a$ from $0$ to $2d$.
The cases $a=0$, $a=d$, and $a=2d$ correspond to $C$ being an $MDS$ code, a \MMDS code, 
and the complement of an $MDS$ code.

\begin{proposition}\label{p:cr}
(a) A  $2$-partition $(C,\overline C)$ of $\VV D(m,n)$ is equitable with the quotient matrix
$\left[ {0 \atop 2m+n}\ {6m+3n \atop 4m+2n}  \right]$ 
(or $\left[ {4m+2n \atop 6m+3n}\ {2m+n \atop 0}  \right]$)
if and only if $C$ (respectively, $\overline C$) is an $MDS$ code. 
(b) 
A $2$-partition $(M,\overline M)$ of $\VV D(m,n)$ is equitable with the quotient matrix
$\left[ {2m+n \atop 4m+2n}\ {4m+2n \atop 2m+n}  \right]$ if and only if $M$ is a \MMDS code.
\end{proposition}

\begin{proof}
We will use the formula $|S|=\frac{c}{b+c}|\VV G|$ for the cardinality of the first cell $S$
of an equitable partition of  a graph $G$ with the quotient matrix 
$\left[ {a \atop c}\ {b \atop d}  \right]$.

(a) For the first matrix, we have $a=0$; so, $C$ is an independent set.
Counting its cardinality gives the cardinality of an MDS code.
For the second matrix, similar argument works for $\overline C$.

(b) Count the size $|M|\cdot b$ of the edge boundary and apply Proposition~\ref{p:eb}.
\end{proof}

However, 
the cases $a=0$, $a=d$, and $a=2d$ are not all possible cases.
In $D(m,0)$ (i.e., $d=2m$), the cases $a=0.5d$ and $a=1.5d$ are also feasible; 
that is there exists an equitable partition with the quotient matrix 
$\left[ {m \atop 3m}\ {5m \atop 3m}  \right]$.
The first ($m=1$) such partition 
is shown in Fig.~\ref{f:1533}.

\begin{proposition}\label{p:1533}
Let $(C,\overline C)$ be the partition shown in Fig.~\ref{f:1533}. 
Then 
\begin{eqnarray*}
\bigg(&\Big\{(x_1,...,x_m)\in \VV D(m,0) \,\Big|\, \sum_{i=1}^m x_i \in C\Big\},&\\
&\Big\{(x_1,...,x_m)\in \VV D(m,0) \,\Big|\, \sum_{i=1}^m x_i \in \overline C\Big\}&\bigg) 
\end{eqnarray*}
is an equitable partition with the quotient matrix 
$\left[ {1m \atop 3m}\ {5m \atop 3m}  \right]$.
\end{proposition}
The proof is straightforward.
\begin{figure}
\centering
\sh{nz nz nz zz}{zz zz nz nz}{nz nz zz nz}{zz nz zz nz}
\ \ \ \ 
\begin{tikzpicture}[
scale=0.7, rotate=180,
nz/.style={circle,fill=white,draw=black, 
           inner sep=2.5pt},
xz/.style={circle,fill=black!50!white,draw=black, 
           inner sep=2.5pt},
zz/.style={circle,fill=black,draw=black, 
           inner sep=2.5pt},
scale=0.7]
\begin{scope}
\clip [xslant=-0.577, yscale=0.867] (-2.4,-2.4) -- (-2.4,-1.0) -- (-1.4,-0.0) -- (-1.4,1.0) -- (-0.0,2.4) -- (1.4,2.4) -- (1.4,1.4) -- (2.4,1.4) -- (2.4,-0.0) -- (1.0,-1.4) -- (0.0,-1.4) -- (-1.0,-2.4) -- cycle ;
\draw[xslant=0.577,ystep=.866,xstep=1,draw=black] (-4.9,-2.1) grid (5.4,3.9);
\draw[xslant=-0.577,ystep=9.866,xstep=1,draw=black] (-3.4,-2.1) grid (6.4,3.9);
\draw (-120:2) 
\shpart zz nz nz nz!
++(120:1) \shpart nz zz nz nz!
++(120:1) ++(1,0)  \shpart nz nz zz zz!
++(120:1) \shpart nz zz nz nz!
++(120:1) \shpart nz zz nz nz!;
\end{scope}
\end{tikzpicture}
\caption{An equitable $2$-partition with quotient matrix $\left[ {1 \atop 3}\ {5 \atop 3}  \right]$}
\label{f:1533}
\end{figure}
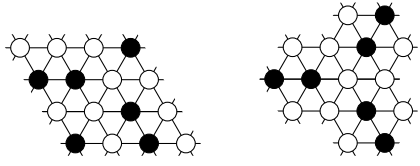

\begin{remark}
Equitable partitions with considered parameters
are connected with unbalanced boolean functions of maximal correlation immunity,
in the sense of the bound \cite{FDF:CorrImmBound}.
A function $f:\VV H(N,2)\to \{0,1\}$ is \emph{unbalanced} if the number of its ones differs from $0$, $2^{N-1}$, and $2^N$,
and is called a $k$th order correlation immune if it has the same number of ones in every subgraph isomorphic to $H(N-k,2)$. 
As was shown in \cite{FDF:CorrImmBound},
the unbalanced boolean functions with maximum correlation immunity order, $2N/3-1$, 
correspond to the equitable partitions of the hypercube graph $H(N,2)$, $N=3d$,
with quotient matrix 
$\left[ {a \atop a+d}\ {3d-a \atop 2d-a}  \right]$, $3d-a \ne a+d$.
Only three classes such equitable partitions are currently known, they have $a=0$, $a=d/2$, and $a=3d/4$ \cite{FDF:12cube.en}.
One can note that in the first two cases,
the quotient matrices are the same as in
Proposition~\ref{p:cr}(a) and  Proposition~\ref{p:1533},
corresponding to the MDS-codes and ``\MMMDS codes''.
Indeed, there is there is a close connection between the considered classes of objects in $D(m,n)$ and in $H(6m+3n,2)$:
with an inverse covering  (locally isomorphic map from $\VV H(6m+3n,2)$ to $\VV D(m,n)$) of $D(m,n)$ by $H(6m+3n,2)$, 
any equitable partition of $D(m,n)$ is mapped to an equitable partition of $H(6m+3n,2)$ with the same quotient matrix.
Such covering can be easily constructed as a group homomorphism mapping a connecting set to a connecting set if 
we consider $H(6m+3n,2)$ as a Cayley graph of $Z_4^{3m}\times Z_2^{3n}$ 
(the connecting set consists of the elements with only one nonzero coordinate, containing $1$ or $3$).
\end{remark}
 
 \begin{remark}
 We briefly discuss a connection with another important concept 
 related to the equitable partitions.
 A set $S$ of vertices of a graph is said to be 
 \emph{completely regular} 
 (often, a \emph{completely regular code}) 
 if the partition of the graph vertices 
 with respect to the distance
 from $S$ is equitable 
 (the quotient matrix is tridiagonal in this case).
 The number of cells different from $S$ 
 in this partition is called 
 the \emph{covering radius} of $S$.
 Trivially, each cell of an equitable $2$-partition
 is a completely regular set of covering radius $1$.
 It is interesting that not only MDS codes and \MMDS codes, 
 but in some cases the connected components of \MMDS codes
 are completely regular.
 As shown in \cite{KLM:arithmetic}, using the Cartesian product, 
 from every completely regular code of covering radius $1$
 one can obtain 
 a completely regular code of arbitrary covering
 radius.
 A component of a decomposable \MMDS code can be
 represented as the Cartesian product of $k$ \MMDS codes,
 where $k$ is from (\ref{eq:can2MDS}).
 As follows from the results of \cite{KLM:arithmetic},
 such a component is completely regular
 if and only if all \MMDS codes
 $M_i$, $i=1,...,k$ in the decomposition (\ref{eq:can2MDS})
 have the same quotient matrices. 
 This happens if and only if
 $2m_1+n_1 = \ldots = 2m_k+n_k$.
 \end{remark}

\section{Conclusion}
We have proven a characterization 
of the distance-$2$ MDS codes,
or maximum independent sets,
in the Doob graphs $D(m,n)$.

For related objects in $D(m,n)$, 
called the \MMDS codes, we have shown the equivalence between 
the connectedness and the indecomposability.
However, the problem of characterization of the \MMDS codes
(namely, those of them that cannot be split into two MDS codes)
remains open.
We have shown that the \MMDS codes are exactly the sets with maximum cut
(edge boundary).

We have noted that the MDS codes and the \MMDS codes in $D(m,n)$ 
are the equitable $2$-partitions with certain quotient matrices.
The eigenvalues of these matrices are the minimum and the maximum eigenvalues
of $D(m,n)$. We have found that in the case $n=0$, there is a third class
of equitable $2$-partitions that corresponds to these eigenvalues.
Characterizing all partitions from this class 
is also an interesting direction for further research.



\providecommand\href[2]{#2} \providecommand\url[1]{\href{#1}{#1}}
  \def\DOI#1{{\small {DOI}:
  \href{http://dx.doi.org/#1}{#1}}}\def\DOIURL#1#2{{\small{DOI}:
  \href{http://dx.doi.org/#2}{#1}}}

\end{document}